\documentclass{amsart}

\usepackage{amsmath}
\usepackage{url}
\usepackage{amsthm}
\usepackage{amsfonts}
\usepackage{hyperref}

\newtheorem{theorem}{Theorem}[section]
\newtheorem{lemma}[theorem]{Lemma}

\newtheorem{definition}[theorem]{Definition}

\newtheorem{prop}{Proposition}[section]
\newtheorem{remark}[prop]{Remark}

\makeatletter \@addtoreset{equation}{section} \makeatother

\newcommand{\h}{\textbf{h}}

\begin{document}

\title[]{The Gauss Map of a Free Boundary Minimal Surface}
\author{Hung Tran}
\address{Department of Mathematics and Statistics,
 Texas Tech University, Lubbock, TX 79409}
\email{hung.tran@ttu.edu}

\renewcommand{\subjclassname}{%
  \textup{2000} Mathematics Subject Classification}
\date{}
\begin{abstract} In this paper, we study the Gauss map of a free boundary minimal surface. The main theorem asserts that if components of the Gauss map are eigenfunctions of the Jacobi-Steklov operator, then the surface must be rotationally symmetric.     

\end{abstract}
\maketitle
\tableofcontents

\section{Introduction}
The goal of this paper is to study the Gauss map of a free boundary minimal surface (FBMS). Let $\Sigma_0$ be an abstract surface which is properly immersed in the unit Euclidean ball $\mathbb{B}^3$. The immersed surface, denoted by $\Sigma$, is a FBMS if its mean curvature vanishes and $\partial \Sigma$ meets $\partial \mathbb{B}^n$ perpendicularly. Equivalently, a FBMS is a critical point of the area functional among all surfaces with boundaries on $\partial \Omega$. 
For example, the only rotationally symmetric embedded smooth FBMS are equatorial planes and critical catenoids (appropriate portions of catenoids which meet the boundary sphere perpendicularly). In higher dimensions, we also consider cones over minimal submanifolds of the boundary sphere. 

The subject of FBMS has an extensive literature much of which has been devoted to study existence and regularity problems (see \cite{DHTKv2} for an excellent survey and \cite{FS16, MNS16, Li_free15, LZminmax16, KWtriple17} for recent results). In particular, for regularity, if $\partial \Sigma$ is real analytic then $\Sigma$ can be continued analytically across the boundary. The surface may still develop some singular set as demonstrated by the minimal cone example. For existence, the work of A. Fraser and R. Schoen shows there is a FBMS with genus zero and any number of boundary components by exploiting an intriguing spectral geometry connection \cite{FS16}. 

Additionally, more attention is directed towards a better understanding of the stability of these critical objects through the second variation. As a consequence, there have been several attempts to compute and estimate the Morse index which intuitively gives the number of distinct admissible deformations which decrease the area to second order
(\cite{Sargent16, ACS16, devyver16index, SZ16index, tran16index, SSTZ17morse}). In particular, in \cite{tran16index}, the author develops a new and natural method to compute the index from data of two simpler problems: the fixed boundary problem and the Jacobi-Steklov problem (Steklov eigenvalues associated with the Jacobi operator). For surfaces which are star-shaped (that applies to the family constructed by Fraser and Schoen and all known explicit examples), the former is well-known and only the latter is non-trivial. 

Generally, the Jacobi-Steklov problem is intrinsically related to the Gauss map since a normal vector field is a Jacobi field. In this paper, we will explore their relationship. First, consider the critical catenoid parametrized by, for an appropriate constant $c$,
\[X(t,\theta)=c(\cosh{t}\cos{\theta}, \cosh{t}\sin{\theta}, t).\]
A normal vector field is given by 
\[\nu=\frac{1}{\cosh{t}}(\cos{\theta}, \sin{\theta},-\sinh{t}).\]
It is straightforward to check that each normal component is an eigenfunction of the Jacobi-Steklov problem (see \cite{tran16index} for details). Nevertheless, it is noted that normal components are extrinsic quantities which depend on the choice of the immersion. For instance, rotating the critical catenoid above by an angle $\alpha$ in the $yz$-plane yields another critical catenoid:
\[ X^\alpha(t, \theta)= c(\cosh{t}\cos{\theta}, \cosh{t}\sin{\theta}\cos{\alpha}-t\sin{\alpha}, \cosh{t}\sin{\theta}\sin{\alpha}+ct\cos{\alpha}).\]
The normal vector $\nu^\alpha$ is obtained by rotating $\nu$ by the same angle
\[\nu^\alpha =\frac{1}{\cosh{t}}(\cos{\theta}, \sinh{t}\sin{\alpha}+\sin{\theta}\cos{\alpha}, \sin{\theta}\sin{\alpha}-\sinh{t}\cos{\alpha}).\]
For this catenoid, however, only certain linear combinations of those normal components are eigenfunctions. Our main theorem asserts that property characterize critical catenoids among non-equatorial FBMS. 
\begin{theorem}
\label{mainthm1}
Let $\Sigma\subset \mathbb{B}^3$ be a smooth, properly immersed FBMS such that the span of normal components is an eigen-subspace of the Jacobi-Steklov problem. Then $\Sigma$ must be rotationally symmetric. 
\end{theorem}
\begin{remark} For the precise definition of those concepts, see Section \ref{notation}.
\end{remark}
We have the following observations. 
\begin{itemize}
\item If we further assume that $\Sigma$ is embedded then it must be either an equatorial disk or a critical catenoid.
\item It is interesting to relate to the theory of minimal surfaces in a sphere. In that case, components of a normal vector of any minimal surface are always eigenfunctions of the corresponding Jacobi operator. 
\item The results are also relevant to a well-known conjecture which says that an embedded FBMS with index of 4 must be the critical catenoid. By \cite{tran16index}, for such an FBMS, the eigen-subspace with eigenvalues smaller than 1 of the Jacobi-Steklov problem is precisely of dimension 3. Furthermore, the associated index form is also negative definite on the span of normal components which is of dimension 3 for non-equatorial surfaces. In relation, Theorem \ref{mainthm1} is based on a slightly stronger assumption that these subspaces coincide. 
\end{itemize}

The organization of the paper is as follows. Section \ref{notation} collects preliminaries including a description of the Jacobi-Steklov problem and a re-formalization as a problem on a spherical domain, which might be of independent interests. Then, in Section \ref{proofmainthm}, we give a proof of Theorem \ref{mainthm1}. In order to do that, we first derive the result under a stronger assumption: namely each normal component is an eigen-function (see Theorem \ref{mainthm}). In that case, recall that a normal vector is a Jacobi field. Thus, when its components are eigenfunctions, they must satisfy further equations along the boundary. Hence, the proof of Theorem \ref{mainthm} is based on a careful analysis of these equations and their consequences. For Theorem \ref{mainthm1}, we use a linear algebra argument to reduce to the case of Theorem \ref{mainthm}. Finally, the Appendix collects a straightforward but crucial computation used in the proof. 

{\bf Acknowledgments:} The author would like to thank Richard Schoen for extensive discussion. The author has also benefited greatly from conversations with Xiaodong Cao, Magdalena Toda, and David Wiygul. 
\section{Notation and Preliminaries}\label{notation}
This section collects notation and preliminary results. Let $\mathbb{B}^3$ be the ball of radius 1 around the origin in $\mathbb{R}^3$ and $\mathbb{S}^2:= \partial \mathbb{B}^3$ is the boundary sphere. Let $(\Sigma_0,g)$ be a smooth abstract surface with boundary. $\Sigma=X(\Sigma_0)$ is an isometric immersion of $\Sigma_0$ in $\mathbb{B}^3$ such that $\partial \Sigma=\Sigma \cap \partial \mathbb{B}^3$. When the context is clear, we identify the image with the surface itself. $\nu:\Sigma \mapsto\mathbb{S}^2$ is a choice of a smooth section of the normal bundle of $\Sigma \subset \mathbb{R}^3$. If $\Sigma$ is non-equatorial then 
$\nu$ is chosen so that \[\zeta:=X \cdot \nu\] is positive at some point. For a choice of normal vector $\nu$, the 2nd fundamental form is defined as, for orthogonal unit tangent vectors $T_i, T_j$,
\[\h_{ij}:=\h(T_i, T_j)=D_{T_i}T_j \cdot \nu=-D_{T_i}\nu \cdot T_j.\]
Here $D_{-}(-)$ is the covariant derivative and ($\cdot$) denotes the inner product in the ambient Euclidean space. Then $|\h|$ denotes the norm. Let $\{e_1,...,e_{3}\}$ be an orthonormal basis of $\mathbb{B}^{3}$ and 
\begin{align*}
X_i &:=X\cdot e_i,\\
\nu_i&:=\nu \cdot e_i.
\end{align*}
Finally, $\eta$ denotes the outward conormal vector along the boundary of a surface. It is noted that if $\Sigma$ is an FBMS in the Euclidean ball around the origin then $\eta\equiv X$ along the boundary.  

\subsection{Jacobi-Steklov Problem}
For $\Delta$ denoting the Laplacian operator, let \[J:= \Delta+|\h|^2\] denote the Jacobi operator. For an FBMS in the ball, the index form associated with the second variation of the area functional is given by
\begin{align}
\label{2ndvarAREA}
{S}(f,h) &:= \int_\Sigma (DfDh-|\h|^2 fh)d\mu -\int_{\partial \Sigma} fh da\\
&=-\int_\Sigma fJh d\mu +\int_{\partial \Sigma} f(D_{\eta}h-h) da \nonumber.
\end{align}

Following \cite{tran16index}, we consider the Jacobi-Steklov problem. Given a function $\hat{f}\in C^{\infty} (\partial \Sigma)$, consider the Jacobi extension of $\hat{f}$:
\[
\Bigg\{ 
\begin{tabular}{cc}
$J f =0$ & \text{ on $\Sigma$},\\
${f}=\hat{f}$ & \text{on $\partial \Sigma$}.
\end{tabular}
\]
It is well-known that an extension exists if and only if $\hat{f}$ is $L^2$-perpendicular to the kernel of $J$ with Dirichlet data. For that domain (see \cite[Section 2]{tran16index} for more details), the Dirichlet-to-Neumann map is defined 
\begin{equation}\label{jacobiSteklov}
 L_{J} \hat{f}=D_{\eta} {f}.
 \end{equation}
It turns out that $L_{J}$ has a discrete spectrum tending to infinity
\begin{definition}
A function $u\in C^{\infty}(\Sigma)$ is said to be an eigenfunction with eigenvalue $\lambda$ of (\ref{jacobiSteklov}) if the following holds:
\begin{equation*}
\begin{cases}
Ju &=0,\\
D_\eta u &= \lambda u.
\end{cases}
\end{equation*}
\end{definition}
\begin{definition}
A subspace $W\subset C^{\infty}(\Sigma)$ is said to be an eigen-subspace of (\ref{jacobiSteklov}) if there is a basis consisting of eigenfunctions of (\ref{jacobiSteklov}).  
\end{definition} 
\begin{lemma}\label{choosebasis}
A finite-dimensional eigen-subspace has a basis $\{u_1,...,u_m\}$ such that, for $i\neq j$, 
\[S(u_i, u_j)=0.\]
\end{lemma}
\begin{proof}
We note that if $u_i, u_j$ are eigenfuntions $D_\eta u_i=\lambda_i u_i$, then
\begin{align*}
S(u_i,u_j) &= -\int_{\Sigma} u_i J u_j d\mu+\int_{\partial\Sigma}u_i (D_\eta u_j-u_j)da,\\
&=(\lambda_j-1)\int_{\partial\Sigma}u_iu_jda\\
&=(\lambda_i-1)\int_{\partial\Sigma}u_iu_jda.
\end{align*} 
So if $\lambda_i\neq \lambda_j$ then, clearly, $S(u_i,u_j)=0$. If $\lambda_i=\lambda_j$ then, by the Gram-Schmidt process, the result also follows. 
\end{proof}

\subsection{The Gauss Map} The Gauss map $\nu:\Sigma \mapsto\mathbb{S}^2$ is an immersion at every point where $|\h|\neq 0$. Let $V:=\text{span}(\nu_1,...,\nu_3)\subset C^{\infty}(\Sigma)$.
\begin{lemma}\label{dimV} If $\Sigma$ is an equatorial disk, then $\text{dim}(V)=1$. Otherwise, $\text{dim}(V)=3$.
\end{lemma}
\begin{proof}
When $\Sigma$ is an equatorial disk, $\nu$ is a constant vector and thus $\text{dim}(V)=1$. Otherwise, it is clear that $\text{dim}(V)\leq 3$. If $\text{dim}(V)<3$ then there is a constant vector $a$ such that $\nu\cdot a=0$. Thus $\nu(\Sigma)$ is in an equator and $\nu$ is nowhere an immersion. As a consequence, $|\h|=0$ and $\Sigma$ is equatorial. Then the statement follows.  
\end{proof}
Next, we collect various equations which follow from the minimality and free boundary conditions. First, the minimality implies that 
\begin{align*}
\Delta X &=0,\\
J \nu &=0,\\
J \zeta &=0.
\end{align*}
Since the surface meets the boundary sphere perpendicularly, the followings hold:
\begin{align*}
D_\eta X_i &=X,\\
D_{\eta} \nu_i &=-\h(\eta, \eta)X_i,\\ 
D_{\eta} \zeta &=-\h(\eta,\eta).
\end{align*}
Furthermore, we have the following integral identities (for a proof, see \cite{FS16, tran16index}). For $i\neq j$,
\begin{align}
\label{integrateuiuj}
S(\nu_i, \nu_j) &= 2\int_{\partial \Sigma}X_i X_j,\\
\label{integratenu2}
S(\nu_i, \nu_i) &=-2\int_{\Sigma}\nu_i^2.
\end{align}
\begin{lemma}\label{choosecoor}
There is a coordinate of $\mathbb{B}^3$ such that, for $i\neq j$,
\[S(\nu_i, \nu_j) = 2\int_{\partial \Sigma}X_i X_j=0.\]
\end{lemma}

\begin{proof}
Consider the 2-tensor 
\[R(e_i, e_j) :=\int_{\partial\Sigma}(X\cdot e_i)(X\cdot e_j).\]
It is clear that $R$ is symmetric. The result then follows from standard linear algebra. 
\end{proof}

 $\Sigma$ is said to have a non-degenerate Gauss map if $\nu$ is an immersion everywhere. In that case, the abstract surface $\Sigma_0$ could be equipped with the pullback by $\nu$ of the round metric. To distinguish with the abstract surface $(\Sigma_0, g)$, that Riemannian surface is denoted by $\Omega=(\Sigma_0, g_1)$. 
\begin{remark}
When $\nu$ is one-to-one, $\Omega$ is isometric to a domain on the sphere. 
\end{remark} 
 Since $\nu$ is conformal, we have
\begin{equation}
\label{conformalg1}
g_1=\frac{|\h|^2}{2}g.
\end{equation}
Consequently, the Laplacian with respect to $g_1$ is related to one with respect to $g$ by
\[\Delta_1 =\frac{2}{|\h|^2} \Delta .\]   

Next we recall the following characterization, due to R. Souam\cite{souam05}, of a FBMS by its Gauss map. The proof will be provided for completeness.  

\begin{lemma}
Let $\Sigma \subset \mathbb{B}^3$ be a FBMS with non-degenerate Gauss map $\nu$.  Then $\Omega=(\Sigma, g_1)$ satisfies the following properties:
\begin{itemize}
\item $2$ is an eigenvalue of the Laplacian with Dirichlet condition.
\item One eigenfunction associated with eigenvalue $2$ has constant boundary derivative. 
\end{itemize}
\end{lemma}
\begin{proof}
Consider the function $\zeta=X\cdot \nu$. Since $\Sigma$ is a FBMS, we have, 
\begin{align*}
(\Delta+|\h|^2) \zeta &=0,\\
D_\eta \zeta &= -\h(\eta,\eta).
\end{align*}
We compute, for $\eta_1$ the outward conormal vector of $\Omega$,
\begin{align*} 
\frac{|\h|^2}{2}(\Delta_1+2)\zeta &= (\Delta+|\h|^2) \zeta=0,\\ 
(D_1)_{\eta_1} \zeta &= \sqrt{\frac{2}{|\h|^2}} D_{\eta} \zeta=\frac{-\h(\eta,\eta)}{|\h(\eta,\eta)|}. 
\end{align*}
The result then follows. 
\end{proof}
\begin{remark} If $\Sigma$ is star-shaped then $2$ is the first eigenvalue of $\Omega$. 
\end{remark}
Conversely, we have the following.
\begin{lemma}
Let $\nu:\Omega=(\Sigma, g_1)\mapsto \mathbb{S}^2$ be an isometric immersion such that 
\begin{itemize}
\item $2$ is an eigenvalue of the Laplacian with Dirichlet condition.
\item One eigenfunction $u$ associated with eigenvalue $2$ has boundary derivative constantly equal to 1.
\end{itemize}
We then identify $\Omega$ with its image $\nu(\Omega)$ on $\mathbb{S}^2$. Then the map
\begin{align*}
X^u: ~\Omega  &\mapsto \mathbb{R}^3,\\
\nu &\mapsto X^u(\nu)=\nabla u(\nu)+u(\nu)\nu.
\end{align*} 
defines a branched minimal surface in $\mathbb{R}^3$ with boundary lying on $\mathbb{S}^2$ and intersecting the sphere perpendicularly. Here $\nabla u(\cdot)$ is identified as a vector in $\mathbb{R}^3$.
\end{lemma}
\begin{proof}
The fact that $X$ is a branched minimal surface follows from a straightforward computation (see \cite{MontielRos90}). A consequence is that $\nu$ is normal to the surface $X^u$.

Furthermore, $|X^u|^2= |\nabla u|^2+ u^2=1$ along the boundary of $\Omega$. So the boundary of $X^u(\Omega)$ lies on $\mathbb{S}^2$. Finally, for $\nu\in \Omega$, the tangent plane at $X^u(\nu)$ of $X$ is given by $\{p\in \mathbb{R}^3, p\cdot \nu =\nu\cdot X(\nu)=u(\nu)\}$. Thus, for $\nu\in \partial \Omega$, the tangent plane contains the origin. That is, the surface $X(\Omega)$ intersects the sphere perpendicularly.   
\end{proof}

Next, we recall the following result of R. Reilly \cite{reilly73}.
\begin{lemma}
\label{gaussimmersion}
If the immersion $X:M^n\mapsto \mathbb{S}^{n+1}$ has non-degenerate Gauss map $\nu$, then the principal curvature of $ \nu(M)\subset \mathbb{S}^{n+1}$ are the reciprocals of those for $X(M)\subset \mathbb{S}^{n+1}$. Furthermore, nondegeneracy of the Gauss map is equivalent to the non-vanishing of all principal curvatures of $(X,M)$.
\end{lemma}
We are now ready to prove the main result of this section, namely a re-formalization of the Jacobi-Steklov problem as a consideration on a spherical domain. 
\begin{theorem}
Let $\Sigma \subset \mathbb{B}^3$ be a FBMS with non-degenerate Gauss map $\nu$. Let $g_1$ be the pullback by $\nu$ of the round metric on $\Sigma$. There is a one-to-one correspondence between Jacobi-Steklov eigenvalues associated with (\ref{jacobiSteklov}) on $\Sigma$ with eigenvalues of the following problem on $\Omega=(\Sigma, g_1)$:
\begin{equation}
\begin{cases}
\label{jacobitransform}
(\Delta_1+2) u &=0,\\
(D_1)_{\eta_1} u &= |\kappa| \lambda u.
\end{cases}
\end{equation}
Here $\eta_1$ is the outward conormal vector along the boundary of $\Omega$ and $\kappa$ is the geodesic curvature of $\partial \Omega\subset \Omega$. 
\end{theorem}
\begin{proof}
Let $u$ be an eigenfunction corresponding to a Jacobi-Steklov eigenvalue $\lambda$. Then,
\begin{align*}
(\Delta+|\h|)^2 &= 0,\\
D_{n} u &= \lambda u.
\end{align*}
By Equation \ref{conformalg1}, 
\begin{align*}
(\Delta_1+2) u &= \frac{2}{|\h|^2}(\Delta+|\h|^2)\\
 &=0.
\end{align*}
Similarly, 
\begin{align*}
(D_1)_{\eta_1} u &= \frac{\sqrt{2}}{|\h|}D_\eta u\\
&= \frac{\sqrt{2}}{|\h|}\lambda u.
\end{align*}
However, $\frac{|\h|}{\sqrt{2}}=|\h(X,X)|$ is exactly the absolute value of the geodesic curvature of $\partial \Sigma\subset \mathbb{S}^2$. Applying Lemma \ref{gaussimmersion} yields the result. 
\end{proof}

\begin{remark}
When $\Sigma$ is star-shaped then $\zeta=X\cdot \nu$ is the first eigenfunction of the Laplacian on $\Omega$. As a consequence, the boundary integral of any eigenfunction of (\ref{jacobitransform}) is zero. 
\end{remark}

\begin{remark}
(\ref{jacobitransform}) corresponds to the following functional:
\[ S_1(u,u):=\int_{\Omega} |\nabla u|^2-2u^2 -\int_{\partial \Omega} |\kappa| u^2. \]
\end{remark}

\section{Proof of the Main Result}
\label{proofmainthm}
In this section, we will give a proof of Theorem \ref{mainthm1}. First, we will consider a slightly more general setting of an orientable free boundary minimal hypersurface $\Sigma^n\subset \mathbb{B}^{n+1}$. $\Sigma$ is assumed to be real analytic outside a possible singular set away from the boundary. That weakened assumption allows the consideration of minimal cones. For such a hypersurface, a smooth section of the normal bundle is a Jacobi field. Then we consider an additional assumption on each normal component along the boundary: 
\begin{equation}
\label{eigenequation}
D_{\eta}\nu_i =\lambda_i \nu_i~~\forall i.
\end{equation}
\begin{remark}
If $\Sigma$ is without a singular set then assumption \ref{eigenequation} is equivalent to the condition that each normal component is an eigenfunction of (\ref{jacobiSteklov}). 
\end{remark}
Recall that 
\begin{equation}
\label{bdryeqn}
D_{\eta}\nu_i =-\h(\eta,\eta)X_i.
\end{equation}

For $f:=-\h(\eta, \eta)$, assumption \ref{eigenequation} yields:
\begin{align}
fX_i \nu_i &= \lambda_i \nu_i^2,\\
f^2X_i^2 &= \lambda_i^2 \nu_i^2,\\
0 &=\sum_{i} \lambda_i \nu_i^2,\\
\label{fsquare}
f^2 &= \sum_i \lambda_i^2 \nu_i^2.
\end{align}

\begin{lemma}\label{onelambdaiszero} 
$\lambda_i=0$ for some $i$ if and only if $\Sigma$ is a minimal cone. 
\end{lemma}
\begin{proof}
If $\lambda_i=0$, then by (\ref{eigenequation}) and (\ref{bdryeqn}), $f X_i=0$. Since $\Sigma$ is real analytic along the boundary, either $f\equiv 0$ or $X_i \equiv 0$ along the boundary. 

If $f\equiv 0$, then the mean curvature of $\partial \Sigma\subset \partial \mathbb{B}^{n+1}$ is zero. In that case, $\partial\Sigma$ must be a minimal surface in the boundary sphere. By the unique continuation of the minimal free boundary condition, $\Sigma$ must be a minimal cone. 

If $X_i \equiv 0$, then $\partial\Sigma$ must be the intersection of the boundary sphere and an hyper-plane. Thus, $\partial\Sigma$ must be round. By the unique continuation again, $\Sigma$ must be rotationally symmetric. Hence, $\Sigma$ is either a free boundary catenoid or an equator. Direct computation, see \cite{SSTZ17morse}, then eliminates the former. 
\end{proof}

\begin{lemma}
\label{equalthen0}
If $\lambda_1=\lambda_2...=\lambda_n=\lambda$, then $\lambda=0$. 
\end{lemma}
\begin{proof}
By equation (\ref{fsquare}), the assumption here implies $f$ is a constant. If the constant is non-zero then equations (\ref{eigenequation}) and (\ref{bdryeqn}) imply $X_i=c\nu_i ~\forall i$ for $c=\pm 1$ , a contradiction to the fact that $\sum_i X_i \nu_i =0$ and $\sum_i \nu_i^2=1$.   
\end{proof}

Now we restrict to the case $\Sigma$ is of dimension two. 
\begin{lemma}
\label{nolambdaiszero}
If $\lambda_i\neq 0$ for each $i$, then $\Sigma$ must be rotationally symmetric. 
\end{lemma}

\begin{proof}
First, we observe that the assumption implies $f^2>0$ since otherwise, by equation (\ref{fsquare}), $\lambda_i=0$ for some $i$. Next, we consider the map $F: \mathbb{S}^2=\partial \mathbb{B}^3\mapsto \mathbb{R}^3$ given by 
\[F(s_1,...,s_3)=(s_1^2,...,s_3^2).\]
It is obvious that $F$ is a local diffeomorphism around each point not on the coordinate planes. Furthermore, we have that
\begin{align}
\frac{X_i}{\lambda_i} &= \frac{\nu_i}{f},\\
\frac{X_i^2}{\lambda_i} &= \frac{\nu_i\cdot X_i}{f}=0,\\
\label{inversefsquare}
\sum_i\frac{X_i^2}{\lambda_i^2} &= \frac{1}{f^2}.
\end{align} 

Thus, $F(\partial \Sigma)$ is in the plane $x+y+z=1$ and also perpendicular to the constant vector $\vec{\lambda}=(\frac{1}{\lambda_1},...,\frac{1}{\lambda_3})$. By Lemma \ref{equalthen0}, the assumption $\lambda_i\neq 0 ~~\forall i$ implies that $\vec{\lambda}$ is not perpendicular to the plane $x+y+z=1$. Thus, each connected component of  $F(\partial \Sigma)$ is a line segment. Thus, without loss of generality, choose a point $p\in \partial \Sigma_0$ not on any coordinate plane then around $F(p)$, for some parameter $t$,
\[F(\partial \Sigma) = (a_1 t+ b_1,..., a_3 t+b_3).\]    
Then, in a neighborhood of $p$, 
\begin{align*}
X_i &= \sqrt{a_i t +b_i}.
\end{align*}
Thus, by Lemma \ref{computecurvature}, the curvature of $\partial \Sigma$ is given by
\begin{equation}
\label{firsteqkappa}
\kappa^2 = 1+\frac{C}{(At+B)^3}.
\end{equation}
On the other hand, 
\[ \kappa^2=\kappa_g^2+\kappa_n^2,\]
where $\kappa_g$ and $\kappa_n$ are the geodesic and normal curvature of the curve with respect to $\partial \mathbb{B}^3$. In our case, $\partial \Sigma$ can be parametrized by arc length with tangent vector $T$. Then 
\begin{align*}
\kappa_g &=D_T T \cdot \nu\\
& =\h(T,T)=-\h(\eta, \eta),\\
\kappa_n &=D_T T\cdot X=\h^{\partial \mathbb{B}^3}(T,T)=-1. 
\end{align*}
Therefore, using equation (\ref{inversefsquare}) yields
\begin{equation}
\label{2ndeqkappa}
\kappa^2 = 1+f^2= 1+\frac{1}{Pt+Q}.
\end{equation}
Comparing (\ref{firsteqkappa}) and (\ref{2ndeqkappa}), we obtain that $A=0$. Thus, Lemma \ref{computecurvature} implies that $a_i=0$ for some $i$. As a consequence, $X_i$ must be constant along each connected component of $\partial \Sigma$. Thus, each connected component of $\partial \Sigma_0$ must be the intersection of a plane and $\partial \mathbb{B}^3$ and, thus, round. The result then follows. 
\end{proof}
The next theorem characterizes an FBMS whose normal components are eigenfunctions of the Jacobi-Steklov problem. 
\begin{theorem}
\label{mainthm}
Let $\Sigma\subset \mathbb{B}^3$ be a smooth, properly immersed FBMS such that each component of a normal vector is an eigenfunction of the Jacobi-Steklov problem. Then $\Sigma$ must be rotationally symmetric. 
\end{theorem}
\begin{proof} For $\Sigma\subset \mathbb{B}^3$, $\partial \Sigma$ is minimal inside $\partial \mathbb{B}^3$ if and only if it is a multiple of the equator. Thus, the minimal cone is actually smooth and rotationally symmetric. The result then follows from Lemma \ref{onelambdaiszero} and Lemma \ref{nolambdaiszero}.
\end{proof}

We are now ready to prove the main result. 
\begin{proof} (\textbf{of Theorem \ref{mainthm1}}) Let $V=\text{span}(\nu_1,...,\nu_3)$ and we assume that $\text{dim}(V)=3$. Otherwise, by Lemma \ref{dimV}, $\Sigma$ must be flat and we are done. 

By Lemma \ref{choosebasis}, $V$ has a basis of eigenfunctions $\{u_1,...,u_3\}$ such that, for $i\neq j$, 
\[S(u_i, u_j)=0.\]
Furthermore, by rescaling if necessary, we have $\forall i$,
\begin{align*}
u_i &= a_{ij}\nu_j:= \nu\cdot a_i,\\
|a_i|&=1. 
\end{align*}
Consider the matrix 
\[M(e_i,e_j) := S(\nu_i,\nu_j).\]
It is clear that $M$ is a symmetric matrix which is negative definite. In particular, it has 3 eigenvectors and 3 negative eigenvalues. For $i\neq j$,
\[M(a_i, a_j)=S(\nu\cdot a_i, \nu\cdot a_j)=0\]
Thus the $a_i$'s are eigenvectors of $M$ and, as a result, for $i\neq j$,
\[a_i\cdot a_j=0\]
Using $\{a_1,...,a_3\}$ as the new basis of $\mathbb{R}^3$ and applying Theorem \ref{mainthm} yield the result. 

\end{proof} 
\begin{remark}
The analogous problem in higher dimensions, classifying free boundary minimal hypersurface $\Sigma^n\subset \mathbb{B}^{n+1}$ satisfying (\ref{eigenequation}), is still open. Note that our proof for $\Sigma^2\subset \mathbb{B}^3$ depends crucially on the computation of the curvature of a curve in space. In higher dimensions, the existence of minimal cones which are not rotationally symmetric indicates the intricacy of the situation.  
\end{remark}
\section{Appendix}
In this section, we study the map $F: \mathbb{S}^2\mapsto \mathbb{R}^3$ given by 
\[F(s_1,...,s_3)=(s_1^2,...,s_3^2).\]
It is noted that $F$ maps the unit sphere onto the triangle obtained by the intersection of the plane $x+y+z=1$ and the cone $\{x\geq 0, y\geq 0, z\geq 0\}$. Also $F$ is a local diffeomorphism around each point not on the coordinate planes. 

Here we are mostly interested in the case when the image of a curve on $\mathbb{S}^2$ is a line segment. As a consequence, locally the curve has the following parametrization. Let $I$ be an interval on the real line and, for $t\in I$ let $\gamma(t)=(\sqrt{a_1 t+b_1},...,\sqrt{a_3 t+b_3})$ be a curve such that 
\begin{align}
\label{sumai}
\sum_ i a_i &=0,\\
\label{sumbi}
\sum_i b_i&=1.
\end{align}
The following computation is crucial in the proof of Lemma \ref{nolambdaiszero}.
\begin{lemma}
\label{computecurvature}
The curvature of $\gamma(t)$ is given by
\[ \kappa(t)^2 =1+\frac{C}{(At+B)^3},\]
for some constant $A, B, C$. 
\end{lemma}
\begin{proof}
Recall that 
\[\kappa =\frac{||\gamma'\times \gamma''||}{||\gamma'||^3}.\]
We compute:
\begin{align*}
2\gamma'(t) &= (\frac{a_1}{\sqrt{a_1 t+b_1}},...,\frac{a_3}{\sqrt{a_3 t+b_3}}),\\
-4\gamma''(t) &= (\frac{a_1^2}{(a_1 t+b_1)^{3/2}},...,\frac{a_3}{(a_3 t+b_3)^{3/2}}). 
\end{align*}
Thus, 
\begin{align*}
-8 \gamma'\times \gamma'' =& \Big( \frac{a_2 a_3(a_3b_2-a_2b_3)}{(a_2t+b_2)^{3/2}(a_3t+b_3)^{3/2}}, \frac{a_3 a_1(a_1b_3-a_3b_1)}{(a_1t+b_1)^{3/2}(a_3t+b_3)^{3/2}},\\
& \frac{a_1 a_2(a_2b_1-a_1b_2)}{(a_2t+b_2)^{3/2}(a_1t+b_1)^{3/2}}\Big)
\end{align*}
For
\begin{align*}
A &:= -a_1a_2a_3,\\
B &:= a_1^2 b_2b_3+a_2^2b_1b_3+a_3^2 b_1b_2,
\end{align*}
using equations (\ref{sumai}) and (\ref{sumbi}) yields
\begin{align*}
4||\gamma'(t)||^2 &=\frac{a_1^2(a_2t+b_2)(a_3t+b_3)+a_2^2(a_1t+b_1)(a_3t+b_3)+a_3^2(a_1t+b_1)(a_2t+b_2)}{(a_1t+b_1)(a_2t+b_2)(a_3t+b_3)}\\
&= \frac{A t + B}{(a_1t+b_1)(a_2t+b_2)(a_3t+b_3)}.
\end{align*}
Therefore,
\begin{align*}
\kappa^2 &= \frac{A_1 t^3+ A_2 t^2+A_3 t+ A_4}{(At+B)^3}.
\end{align*}
The coefficients $A_1,..., A_4$ are computed below. Note that we repeatedly make use of equations (\ref{sumai}) and (\ref{sumbi}). 
\begin{align*}
A_1 &= a_2^2a_3^2(a_3b_2-a_2b_3)^2 a_1^3+ a_3^2a_1^2(a_1b_3-a_3b_1)^2 a_2^3+ a_1^2a_2^2(a_2b_1-a_1b_2)^2 a_3^3\\
&=A^2 \Big(a_1(a_3b_2-a_2b_3)^2+a_2(a_1b_3-a_3b_1)^2+a_3(a_2b_1-a_1b_2)^2\Big)\\
&=A^3(b_1+b_2+b_3)^2\\
&=A^3.
\end{align*}
Next,
\begin{align*}
A_2 &=3a_2^2a_3^2(a_3b_2-a_2b_3)^2 a_1^2b_1+ 3a_3^2a_1^2(a_1b_3-a_3b_1)^2 a_2^2b_2+3a_1^2a_2^2(a_2b_1-a_1b_2)^2 a_3^2b_3\\
&=3A^2 \Big(b_1(a_3b_2-a_2b_3)^2+b_2(a_1b_3-a_3b_1)^2+b_3(a_2b_1-a_1b_2)^2\Big)\\
&=3A^2\Big(a_1^2 b_2b_3(b_2+b_3)+a_2^2 b_1b_3(b_1+b_3)+a_3^2b_1b_2(b_1+b_2)\\
&-2b_1b_2b_3(a_1a_2+a_2a_3+a_3a_1)\Big)\\
&=3A^2 B. 
\end{align*}
Next,
\begin{align*}
A_3 &=3a_2^2a_3^2(a_3b_2-a_2b_3)^2 a_1b_1^2+ 3a_3^2a_1^2(a_1b_3-a_3b_1)^2 a_2b_2^2+3a_1^2a_2^2(a_2b_1-a_1b_2)^2 a_3b_3^2\\
&=-3A\Big(b_1^2a_2a_3(a_3b_2-a_2b_3)^2+b_2^2 a_1a_3(a_1b_3-a_3b_1)^2+b_3^2a_1a_2(a_2b_1-a_1b_2)^2\Big)\\
&=-3A\Big(b_1^2b_2^2a_3^3(a_1+a_2)+b_1^2b_3^2a_2^3(a_1+a_3)+b_2^2b_3^2a_1^3(a_2+a_3)\\
&-2b_1b_2b_3(b_3a_1^2a_2^2+b_1a_2^2a_3^2+b_2a_3^2a_1^2)\Big)\\
&=3A B^2. 
\end{align*}
The result then follows for
\begin{align*}
C &:=A_4-B^3,\\
A_4 &=a_2^2a_3^2(a_3b_2-a_2b_3)^2 b_1^3+ a_3^2a_1^2(a_1b_3-a_3b_1)^2 b_2^3+a_1^2a_2^2(a_2b_1-a_1b_2)^2 b_3^3.
\end{align*}
\end{proof}

\begin{remark} By direct computation, $C=0$ if and only if $\gamma(t)$ is a great circle perpendicular to a coordinate plane. Also $\kappa(t)^2$ is a constant if and only if $\gamma(t)$ is a circle perpendicular to a coordinate plane. In other words, other circles do not map to a line segment.
\end{remark}


\def\cprime{$'$}
\bibliographystyle{plain}
\bibliography{bio}

\end{document}